\documentclass[12pt]{article}
\usepackage{amssymb}
\usepackage{amsthm}
\usepackage{amsmath}
\usepackage{graphicx}
\usepackage{enumerate}
\usepackage{pict2e}
\usepackage{framed}

\newtheorem{theorem}{Theorem}[section]

\newtheorem{lemma}[theorem]{Lemma}

\newtheorem{remark}[theorem]{Remark}
\newtheorem{example}[theorem]{Example}

\title{A characterization of permutability of $2$-uniform tolerances on posets}
\author{Ivan~Chajda and Helmut~L\"anger}
\date{}

\begin{document}

\footnotetext{Support of the research of the first author by the Czech Science Foundation (GA\v CR), project 24-14386L, entitled ``Representation of algebraic semantics for substructural logics'', and by IGA, project P\v rF~2024~011, and of the second author by the Austrian Science Fund (FWF), project I~4579-N, entitled ``The many facets of orthomodularity'', is gratefully acknowledged.}

\maketitle

\begin{abstract}
Tolerance relations were investigated by several authors in various algebraic structures, see e.g.\ the monograph \cite C. Recently G.~Cz\'edli \cite{C22} studied so-called $2$-uniform tolerances on lattices, i.e.\ tolerances that are compatible with the lattice operations and whose blocks are of cardinality $2$. He showed that two such tolerances on a lattice containing no infinite chain permute if and only if they are amicable (a concept introduced in his paper). We extend this study to tolerances on posets. Since in posets we have no lattice operations, we must modify the notion of amicability. We modified it in such a way that in case of lattices it coincides with the original definition. With this new definition we can prove that two tolerances on a poset containing no infinite chain permute if and only if they are amicable in the new sense.
\end{abstract}



\section{Introduction}

By a {\em tolerance} on an algebra $\mathbf A=(A,F)$ is meant a reflexive and symmetric binary relation on $A$ having the Substitution Property with respect to all operations of $F$. For the theory of tolerances see e.g.\ the monograph \cite C. Important are in particular tolerances on lattices since G.~Cz\'edli showed in \cite{C82} that every lattice $\mathbf L=(L,\vee,\wedge)$ can be factorized in a natural way by any tolerance $T$ on $\mathbf L$, i.e.\ the set $L/T$ of all blocks of $T$ forms a lattice again. Such a situation is rather exceptional and does not hold for other types of algebras in general, but many varieties whose members have this property were described in \cite{CCH}.

Recall that a {\em block} of a tolerance $T$ on an algebra $(A,F)$ is a maximal subset $B$ of $A$ satisfying $B^2\subseteq T$.

Recently G.~Cz\'edli proved in \cite{C22} that so-called $2$-uniform tolerances on a lattice containing no infinite chain permute if and only if they are amicable.

The concept of a tolerance on a lattice was generalized to posets by the present authors \cite{CL}. Despite the fact that we do not have explicit operations on a poset, the concept was defined in a way that the blocks are convex again and that posets can be factorized by these tolerances. For other interesting properties of tolerances on posets the reader is referred to \cite{CL}. Hence the natural question arises if also the result of G.~Cz\'edli concerning the permutability of $2$-uniform tolerances can be extended to posets. This is the topic of the present paper.

\section{Basic concepts}

At first we recall several concepts introduced in \cite{CL} and \cite{C22}.

Let $\mathbf P=(P,\leq)$ be a poset. A {\em tolerance} on $\mathbf P$ is a reflexive and symmetric binary relation $T$ on $P$ satisfying the following conditions:
\begin{enumerate}[(1)]
	\item If $(x,y),(z,u)\in T$ and $x\vee z$ and $y\vee u$ exist then $(x\vee z,y\vee u)\in T$.
	\item If $(x,y),(z,u)\in T$ and $x\wedge z$ and $y\wedge u$ exist then $(x\wedge z,y\wedge u)\in T$.
	\item If $x,y,z\in P$ and $(x,y),(y,z)\in T\neq P^2$ then there exist $u,v\in P$ with $u\leq x,y,z\leq v$ and $(u,y),(y,v)\in T$.
	\item If $(x,y)\in T\neq P^2$ then there exists some $(z,u)\in T$ with both $z\leq x,y\leq u$ and $(v,z),(v,u)\in T$ for all $v\in P$ with $(v,x),(v,y)\in T$.
\end{enumerate}
Conditions (3) and (4) are quite natural since they are satisfied by every tolerance on a lattice. In condition (3) one can take $u:=x\wedge y\wedge z$ and $v:=x\vee y\vee z$, and in condition (4) one can take $z:=x\wedge y$ and $u:=x\vee y$.

A {\em block} of a tolerance $T$ on $\mathbf P$ is a maximal subset $B$ of $P$ satisfying $B^2\subseteq T$. Let $P/T$ denote the set of all blocks of $T$. Clearly, $T=\bigcup\limits_{B\in P/T}B^2$. In \cite{CL} we proved that every block of $T$ is convex. According to Zorn's Lemma every subset $A$ of $P$ satisfying $A^2\subseteq T$ is contained in a block of $T$. Following \cite{C22} a tolerance is called {\em $2$-uniform} if every of its blocks consists of exactly two elements. Let $a,b\in P$. We call $a$ a {\em lower $T$-neighbor of $b$} and $b$ an {\em upper $T$-neighbor of $a$} if $a\prec b$ and $(a,b)\in T$. 
 
\section{A characterization of permutability of $2$-uniform tolerances on posets}

It is evident that $2$-uniform tolerances on a poset are very specific. Their basic properties are as follows.

\begin{lemma}\label{lem1}
	Let $\mathbf P=(P,\le)$ be a poset, $a,b\in P$ and $T$ be a $2$-uniform tolerance on $\mathbf P$. Then the following holds:
	\begin{enumerate}[{\rm(i)}]
		\item The element $a$ has at most one lower $T$-neighbor and at most one upper $T$-neighbor,
		\item if $(a,b)\in T$ then $a=b$ or $a\prec b$ or $b\prec a$.
	\end{enumerate}
\end{lemma}

\begin{proof}
	Without loss of generality, assume $|P|>2$. Then $T\ne P^2$.
	\begin{enumerate}[(i)]
		\item Assume $a$ to have two distinct lower $T$-neighbors $c$ and $d$. Then $(c,a),(a,d)\in T$. According to (3) there exists some $e\in P$ with $e\le c,a,d$ and $(e,a)\in T$. Let $B\in P/T$ with $e,a\in B$. Since $e\le c,d\le a$ and $B$ is convex we conclude $c,d,a\in B$ contradicting $|B|=2$. This shows that $a$ can have at most one lower $T$-neighbor. The second assertion follows by duality.
		\item Assume $(a,b)\in T$. First suppose $a\parallel b$. According to (4) there exist $c,d\in P$ with $(c,d)\in T$ and $c\le a,b\le d$. Let $B\in P/T$ with $c,d\in B$. Since $B$ is convex we conclude $c,a,b\in B$. Now we have $a\parallel b$ and $c\le a,b$ and hence $c\ne a,b$ contradicting $|B|=2$. This shows $a\le b$ or $b\le a$. Let $C\in P/T$ with $a,b\in C$. Since $C$ is convex we obtain $[a,b]\subseteq C$ if $a\le b$ and $[b,a]\subseteq C$ if $b\le a$. Now $|C|=2$ implies $a=b$ or $a\prec b$ or $b\prec a$.
	\end{enumerate}
\end{proof}

Let $\mathbf P$ be a poset, $a,b\in P$ and $T,S$ $2$-uniform tolerances on $\mathbf P$.

In what follows we adopt some key concepts from \cite{C22} for $2$-uniform tolerances on posets.

The element $a$ is called a {\em split $(T,S)$-bottom} if there exists some upper $T$-neighbor $b$ of $a$ and some upper $S$-neighbor of $a$ being different from $b$. Further, $a$ is called an {\em adherent $(T,S)$-bottom} if there exists some common upper $T$-neighbor and $S$-neighbor of $a$. Finally, $a$ is called a {\em $(T,S)$-bottom} if it is either a split $(T,S)$-bottom or an adherent $(T,S)$-bottom. The notions {\em split $(T,S)$-top}, {\em adherent $(T,S)$-top} and {\em $(T,S)$-top} are defined dually. A {\em $T$-top} ({\em $T$-bottom}) is an upper (a lower) $T$-neighbor of some element of $P$.

The tolerances $T$ and $S$ are called {\em amicable} if the following four conditions hold:
\begin{enumerate}
	\item[(5)] If $a\ne b$ and there exists some lower $T$-neighbor of $a$ being a lower $S$-neighbor of $b$ then there exists some upper $S$-neighbor of $a$ being an upper $T$-neighbor of $b$.
	\item[(6)] If $a\ne b$ and there exists some upper $T$-neighbor of $a$ being an upper $S$-neighbor of $b$ then there exists some lower $S$-neighbor of $a$ being a lower $T$-neighbor of $b$.
	\item[(7)] if $a$ is a $(T,S)$-top and $b$ is either an upper $T$-neighbor or an upper $S$-neighbor of $a$ then $b$ is a $(T,S)$-top, too,
	\item[(8)] if $a$ is a $(T,S)$-bottom and $b$ is either a lower $T$-neighbor or a lower $S$-neighbor of $a$ then $b$ is a $(T,S)$-bottom, too.
\end{enumerate}
Let us note that any two $2$-uniform tolerances on a lattice $(L,\vee,\wedge)$ satisfy (5) and (6). Namely, G.~Cz\'edli \cite{C22} showed that if $a\ne b$ and there exists some lower $T$-neighbor of $a$ being a lower $S$-neighbor of $b$ then $a\vee b$ is an upper $S$-neighbor of $a$ being an upper $T$-neighbor of $b$, and if $a\ne b$ and there exists some upper $T$-neighbor of $a$ being an upper $S$-neighbor of $b$ then $a\wedge b$ is a lower $S$-neighbor of $a$ being a lower $T$-neighbor of $b$.

\begin{remark}\label{rem1}
	Note that {\rm(7)} and {\rm(8)} coincide with conditions {\rm(A1)} and {\rm(A2)} from {\rm\cite{C22}}, respectively. Hence, for lattices the concept of amicable $2$-uniform tolerances as defined in {\rm\cite{C22}} coincides with the one defined above.
\end{remark}

The proof of our main theorem will broadly use the ideas from \cite{C22}.

\begin{theorem}\label{th1}
	Let $T$ and $S$ be $2$-uniform tolerances on a poset $(P,\le)$ containing no infinite chain. Then $T$ and $S$ permute if and only if they are amicable.
\end{theorem}

\begin{proof}
	Let $a,b\in P$. First assume $T$ and $S$ to permute.
	\begin{enumerate}
		\item[(5)] Suppose $a\ne b$ and suppose $c$ to be some lower $T$-neighbor of $a$ being a lower $S$-neighbor of $b$. Because of (ii) of Lemma~\ref{lem1} we have $a\parallel b$ and hence $(a,b)\notin S\cup T$ again according to (ii) of Lemma~\ref{lem1}. Now $(a,b)\in T\circ S=S\circ T$. Hence there exists some $S$-neighbor $d$ of $a$ being a $T$-neighbor of $b$. Because of $a\parallel b$ we have $d<a,b$ or $d>a,b$. Now $d<a,b$ would imply $(a,b)=(a\vee d,c\vee b)\in T$ according to (1) since $(a,c),(d,b)\in T$. But this contradicts (ii) of Lemma~\ref{lem1}. Therefore $d$ is an upper $S$-neighbor of $a$ being an upper $T$-neighbor of $b$ proving (5).
		\item[(6)] follows from (5) by duality.
		\item[(7)] Suppose $b$ to be either an upper $T$-neighbor or $S$-neighbor of $a$. Without loss of generality assume $(a,b)\in T$. First suppose $a$ to be a split $(T,S)$-top. Let $c$ and $d$ denote the lower $T$-neighbor and $S$-neighbor of $a$, respectively. Then $c\parallel d$. Since $(d,b)\in S\circ T=T\circ S$ there exists some $e\in P$ with $(d,e)\in T$ and $(e,b)\in S$. Now $b\le e$ would imply $d<a<b\le e$ which together with $(d,e)\in T$ would contradict (ii) of Lemma~\ref{lem1}. Again according to Lemma~\ref{lem1} we have $e\prec b$. Now $a$ and $e$ are a lower $T$-neighbor and $S$-neighbor of $b$, respectively, showing $b$ to be a $(T,S)$-top. Next assume $a$ to be an adherent $(T,S)$-top. Let $c$ denote the common lower $T$-neighbor and $S$-neighbor of $a$. Since $(c,b)\in S\circ T=T\circ S$ there exists some $d\in P$ with $(c,d)\in T$ and $(d,b)\in S$. Now $d\le c$ would imply $d\le c<a<b$ which together with $(d,b)\in S$ would contradict (ii) of Lemma~\ref{lem1}, and $b\le d$ would imply $c<a<b\le d$ which together with $(c,d)\in T$ would again contradict (ii) of Lemma~\ref{lem1}. Thus a further application of (ii) of Lemma~\ref{lem1} yields $c\prec d\prec b$. Now $a$ and $d$ are a lower $T$-neighbor and $S$-neighbor of $b$, respectively, showing $b$ to be a $(T,S)$-top.
		\item[(8)] follows from (7) by duality.
	\end{enumerate}
Hence $T$ and $S$ are amicable. \\
Conversely, assume $T$ and $S$ to be amicable. Suppose $(a,b)\in T\circ S$. Then there exists some $c\in P$ with $(a,c)\in T$ and $(c,b)\in S$. We want to show $(a,b)\in S\circ T$. If $a=b$ then $(a,b)=(a,a)\in S\circ T$. If $c=a$ then $(a,b)=(c,b)\in S$ and $(b,b)\in T$ and hence $(a,b)\in S\circ T$. If $c=b$ then $(a,a)\in S$ and $(a,b)=(a,c)\in T$ and hence $(a,b)\in S\circ T$. So we can assume that $a,b,c$ are mutually distinct. According to (ii) of Lemma~\ref{lem1} we consider the following four cases. \\
$\underline{c\prec a\text{ and }c\prec b.}$ \\
Because of (5) there exists a common upper $S$-neighbor $d$ of $a$ and $T$-neighbor of $b$ and hence $(a,d)\in S$ and $(d,b)\in T$ showing $(a,b)\in S\circ T$. \\
$\underline{a\prec c\text{ and }b\prec c.}$ \\
This case is dual to the previous one. \\
$\underline{a\prec c\prec b.}$ \\
Put $a_0:=a$, $a_1:=c$ and $a_2:=b$. For $i\ge3$ define $a_i$ as follows: If $i$ is odd and $a_{i-1}$ is a $T$-bottom then define $a_i$ to be the unique upper $T$-neighbor of $a_{i-1}$. If $i$ is even and $a_{i-1}$ is an $S$-bottom then define $a_i$ to be the unique upper $S$-neighbor of $a_{i-1}$. Note that $a_0$ is a $T$-bottom, $a_1$ is the unique upper $T$-neighbor of $a_0$, $a_1$ is an $S$-bottom and $a_2$ is the unique upper $S$-neighbor of of $a_1$. Since $a_2\prec a_3\prec a_4\prec\cdots$, but $\mathbf P$ has no infinite chain, there exists some $n\ge2$ such that $a_2,\ldots,a_n$ are defined, but $a_{n+1}$ is not. First assume $n$ to be even. Then $a_n$ is not a $T$-bottom. Since every element of $P$ belongs to at least one $2$-element block of $T$, $a_n$ is a $T$-top. But it is also an $S$-top and therefore a $(T,S)$-top. If $n$ is odd then a similar reasoning shows that $a_n$ is a $(T,S)$-top. So $a_n$ is a $(T,S)$-top in any case. Now assume $a_n$ to be an adherent $(T,S)$-top. Then, according to (i) of Lemma~\ref{lem1}, $a_{n-1}$ is an adherent $(T,S)$-bottom. Now (8) shows that $a_{n-2}$ is a $(T,S)$-bottom, too. Next assume $a_n$ to be a split $(T,S)$-top. Suppose $n$ is even. Then $a_{n-1}$ is the unique lower $S$-neighbor of $a_n$, $a_n$ has a unique lower $T$-neighbor $d$ and $a_{n-1}\ne d$. According to (6) there exists some lower $S$-neighbor $e$ of $d$ being a lower $T$-neighbor of $a_{n-1}$. Since $a_{n-2}$ is a lower $T$-neighbor of $a_{n-1}$, too, (i) of Lemma~\ref{lem1} yields $e=a_{n-2}$. Therefore $a_{n-2}$ is a $(T,S)$-bottom. With roles of $T$ and $S$ interchanged, one can prove that $a_{n-2}$ is a $(T,S)$-bottom also in the case when $n$ is odd. Hence in any case $a_{n-2}$ is a $(T,S)$-bottom. Applying (8) finitely many times yields that $a_0=a$ is a $(T,S)$-bottom, too. First assume $a$ to be a split $(T,S)$-bottom. Then there exists some upper $S$-neighbor $f$ of $a$ being distinct from $c$. According to (5) there exists some upper $S$-neighbor $g$ of $c$ being an upper $T$-neighbor of $f$. Since $b$ is an upper $S$-neighbor of $c$, too, (i) of Lemma~\ref{lem1} yields $g=b$. Hence $(a,f)\in S$ and $(f,b)=(f,g)\in T$ showing $(a,b)\in S\circ T$. Now suppose $a$ to be an adherent $(T,S)$-bottom. Then $c$ is a $(T,S)$-top. According to (7), $b$ is a $(T,S)$-top, too. Hence $b$ has a unique lower $T$-neighbor $h$. Assume $c\ne h$. According to (6) there exists some lower $S$-neighbor $i$ of $h$ being a lower $T$-neighbor of $c$. Since $a$ is a lower $T$-neighbor of $c$, too, (i) of Lemma~\ref{lem1} yields $i=a$. On the other hand, $i=a$ is a lower $S$-neighbor of $h$. Hence $c$ and $h$ are two distinct upper $S$-neighbors of $a$ contradicting (i) of Lemma~\ref{lem1}. This shows $c=h$. Now $(a,c)\in S$ and $(c,b)=(h,b)\in T$ proving $(a,b)\in S\circ T$. \\
$\underline{b\prec c\prec a.}$ \\
This case is dual to the previous one. \\
Hence in any case we have $(a,b)\in S\circ T$. This shows $T\circ S\subseteq S\circ T$. By duality, we obtain $S\circ T\subseteq T\circ S$ and hence $T\circ S=S\circ T$, i.e.\ $T$ and $S$ permute
\end{proof}

\section{Examples}

\begin{example}
	Consider the poset $\mathbf P$ depicted in Figure~1
	
	\vspace*{-3mm}
	
	\begin{center}
		\setlength{\unitlength}{7mm}
		\begin{picture}(4,14)
			\put(2,1){\circle*{.3}}
			\put(2,3){\circle*{.3}}
			\put(2,5){\circle*{.3}}
			\put(2,7){\circle*{.3}}
			\put(1,9){\circle*{.3}}
			\put(3,9){\circle*{.3}}
			\put(1,11){\circle*{.3}}
			\put(3,11){\circle*{.3}}
			\put(2,13){\circle*{.3}}
			\put(2,7){\line(0,-1)6}
			\put(2,7){\line(-1,2)1}
			\put(2,7){\line(1,2)1}
			\put(1,9){\line(0,1)2}
			\put(1,9){\line(1,1)2}
			\put(3,9){\line(-1,1)2}
			\put(3,9){\line(0,1)2}
			\put(2,13){\line(-1,-2)1}
			\put(2,13){\line(1,-2)1}
			\put(1.85,.3){$0$}
			\put(2.4,2.85){$a$}
			\put(2.4,4.85){$b$}
			\put(2.4,6.85){$c$}
			\put(.35,8.85){$d$}
			\put(3.4,8.85){$e$}
			\put(.35,10.85){$f$}
			\put(3.4,10.85){$g$}
			\put(1.85,13.4){$1$}
			\put(1.2,-.75){{\rm Fig.~1}}
			\put(1.2,-1.75){Poset}
		\end{picture}
	\end{center}
	
	\vspace*{8mm}
	
	and put
	\begin{align*}
		T & :=\{0,a\}^2\cup\{a,b\}^2\cup\{c,e\}^2\cup\{d,g\}^2\cup\{f,1\}^2, \\
		S & :=\{0,a\}^2\cup\{a,b\}^2\cup\{c,e\}^2\cup\{d,f\}^2\cup\{g,1\}^2.
	\end{align*}
	Then $T$ and $S$ are permuting $2$-uniform tolerances on $\mathbf P$ since
	\[
	T\circ S=T\cup S\cup\{(0,b),(b,0),(d,1),(f,g),(g,f),(1,d)\}=S\circ T
	\]
	and hence they are amicable according to Theorem~\ref{th1}.
\end{example}

\begin{example}
	Consider the poset $\mathbf P$ visualized in Figure~2
	
	\vspace*{-3mm}
	
	\begin{center}
		\setlength{\unitlength}{7mm}
		\begin{picture}(4,16)
			\put(2,1){\circle*{.3}}
			\put(1,3){\circle*{.3}}
			\put(3,3){\circle*{.3}}
			\put(1,5){\circle*{.3}}
			\put(3,5){\circle*{.3}}
			\put(2,7){\circle*{.3}}
			\put(2,9){\circle*{.3}}
			\put(1,11){\circle*{.3}}
			\put(3,11){\circle*{.3}}
			\put(1,13){\circle*{.3}}
			\put(3,13){\circle*{.3}}
			\put(2,15){\circle*{.3}}
			\put(2,1){\line(-1,2)1}
			\put(2,1){\line(1,2)1}
			\put(1,3){\line(0,1)2}
			\put(1,3){\line(1,1)2}
			\put(3,3){\line(0,1)2}
			\put(2,7){\line(-1,-2)1}
			\put(2,7){\line(1,-2)1}
			\put(2,7){\line(0,1)2}
			\put(2,9){\line(-1,2)1}
			\put(2,9){\line(1,2)1}
			\put(1,11){\line(0,1)2}
			\put(1,11){\line(1,1)2}
			\put(3,11){\line(-1,1)2}
			\put(3,11){\line(0,1)2}
			\put(2,15){\line(-1,-2)1}
			\put(2,15){\line(1,-2)1}
			\put(1.85,.3){$0$}
			\put(.35,2.85){$a$}
			\put(3.4,2.85){$b$}
			\put(.35,4.85){$c$}
			\put(3.4,4.85){$d$}
			\put(2.4,6.85){$e$}
			\put(2.4,8.85){$f$}
			\put(.35,10.85){$g$}
			\put(3.4,10.85){$h$}
			\put(.35,12.85){$i$}
			\put(3.4,12.85){$j$}
			\put(1.85,15.4){$1$}
			\put(1.2,-.75){{\rm Fig.~2}}
			\put(1.2,-1.75){Poset}
		\end{picture}
	\end{center}
	
	\vspace*{8mm}
	
	and put
	\begin{align*}
T & :=\{0,a\}^2\cup\{a,c\}^2\cup\{b,d\}^2\cup\{d,e\}^2\cup\{f,h\}^2\cup\{g,j\}^2\cup\{i,1\}, \\
S & :=\{0,b\}^2\cup\{a,d\}^2\cup\{c,e\}^2\cup\{f,g\}^2\cup\{h,j\}^2\cup\{i,1\}^2.
	\end{align*}
	Then $T$ and $S$ are permuting $2$-uniform tolerances on $\mathbf P$ since
	\begin{align*}
		T\circ S & =T\cup S\cup\{(0,d),(a,b),(a,e),(b,a),(c,d),(d,0),(d,c),(e,a),(f,j),(g,h),(h,g), \\
		& \hspace*{6mm}(j,f)\}=S\circ T
	\end{align*}
	and hence they are amicable according to Theorem~\ref{th1}.
\end{example}

\begin{example}
	Consider the poset $\mathbf P$ depicted in Figure~3
	
	\vspace*{-3mm}
	
	\begin{center}
		\setlength{\unitlength}{7mm}
		\begin{picture}(4,8)
			\put(2,1){\circle*{.3}}
			\put(1,3){\circle*{.3}}
			\put(3,3){\circle*{.3}}
			\put(1,5){\circle*{.3}}
			\put(3,5){\circle*{.3}}
			\put(2,7){\circle*{.3}}
			\put(2,1){\line(-1,2)1}
			\put(2,1){\line(1,2)1}
			\put(1,3){\line(0,1)2}
			\put(1,3){\line(1,1)2}
			\put(3,3){\line(-1,1)2}
			\put(3,3){\line(0,1)2}
			\put(2,7){\line(-1,-2)1}
			\put(2,7){\line(1,-2)1}
			\put(1.85,.3){$0$}
			\put(.35,2.85){$a$}
			\put(3.4,2.85){$b$}
			\put(.35,4.85){$c$}
			\put(3.4,4.85){$d$}
			\put(1.85,7.4){$1$}
			\put(1.2,-.75){{\rm Fig.~3}}
			\put(1.2,-1.75){Poset}
		\end{picture}
	\end{center}
	
	\vspace*{8mm}
	
	and put
	\begin{align*}
		T & :=\{0,a\}^2\cup\{b,d\}^2\cup\{c,1\}^2, \\
		S & :=\{0,b\}^2\cup\{a,c\}^2\cup\{d,1\}^2.
	\end{align*}
	Then $T$ and $S$ do not permute since $(a,b)\in(T\circ S)\setminus(S\circ T)$, and $T$ and $S$ are not amicable since $0$ is a lower $T$-neighbor of $a$ being a lower $S$-neighbor of $b$, but there exists no upper $S$-neighbor of $a$ being an upper $T$-neighbor of $b$.
\end{example}

Authors' addresses:

Ivan Chajda \\
Palack\'y University Olomouc \\
Faculty of Science \\
Department of Algebra and Geometry \\
17.\ listopadu 12 \\
771 46 Olomouc \\
Czech Republic \\
ivan.chajda@upol.cz

Helmut L\"anger \\
TU Wien \\
Faculty of Mathematics and Geoinformation \\
Institute of Discrete Mathematics and Geometry \\
Wiedner Hauptstra\ss e 8--10 \\
1040 Vienna \\
Austria, and \\
Palack\'y University Olomouc \\
Faculty of Science \\
Department of Algebra and Geometry \\
17.\ listopadu 12 \\
771 46 Olomouc \\
Czech Republic \\
helmut.laenger@tuwien.ac.at


\begin{thebibliography}9
	\bibitem C
	I.~Chajda, Algebraic Theory of Tolerance Relations. Palack\'y University Olomouc, Olomouc 1991. ISBN 80-7067-042-8.
	\bibitem{CCH}
	I.~Chajda, G.~Cz\'edli and R.~Hala\v s, Independent joins of tolerance factorable varieties. Algebra Universalis {\bf69} (2013), 83--92.
	\bibitem{CL}
	I.~Chajda and H.~L\"anger, Tolerances on posets. Miskolc Math.\ Notes {\bf24} (2023), 725--736.
	\bibitem{C82}
	G.~Cz\'edli, Factor lattices by tolerances. Acta Sci.\ Math.\ (Szeged) {\bf44} (1982), 35--42.
	\bibitem{C22}
	G.~Cz\'edli, Permuting $2$-uniform tolerances on lattices. J.\ Multiple-Valued Log.\ Soft Computing {\bf39} (2022), 97--104.
\end{thebibliography}
\end{document}